\documentclass{tran-l}
\usepackage[centertags]{amsmath}
\usepackage{amsfonts}
\usepackage{amssymb}
\usepackage{amsthm}
\usepackage{newlfont}
\usepackage{lscape}
\usepackage{graphicx}
\newlength{\defbaselineskip}
\newcommand{\setlinespacing}[1]%
          {\setlength{\baselineskip}{#1 \defbaselineskip}}

 \newcommand{\s}{\subseteq}

 \newcommand{\rl}{\mathbb{R}}

 \newcommand{\N}{\mathbb{N}}

 \newcommand{\bth}{\begin{thm}}
 \newcommand{\brm}{\begin{rem}}
 \newcommand{\erm}{\end{rem}}
 \newcommand{\bcor}{\begin{cor}}
 \newcommand{\ecor}{\end{cor}}
 \newcommand{\ed}{\end{document}}
 \newcommand{\bmb}{\begin{mbs}}
 \newcommand{\emb}{\end{mbs}}
 
 \newcommand{\beq}{\begin{eqnarray*}}
 \newcommand{\eeq}{\end{eqnarray*}}
\theoremstyle{plain}
\newtheorem{thm}{Theorem}
\newtheorem{cor}[thm]{Corollary}
\newtheorem{lem}[thm]{Lemma}

\newtheorem{exmpl}[thm]{Example}

\theoremstyle{definition}
\newtheorem{defn}[thm]{Definition}
\newtheorem{rem}[thm]{Remark}

\numberwithin{equation}{section}

\begin{document}

\title{On linear ternary intersection sequences and their properties}

\author{Mahdi Saleh}
\address{National organization for development of exceptional talents (NODET), Tehran, Iran}
\author{Majid Jahangiri}
\address{School of Mathematics, Institute for Research in Fundamental Sciences (IPM), Tehran, Iran.}
\begin{abstract}
Let $D^+$ be the first octant of the Euclidean space and consider the integral cube grid $G$ in $D^+$. The intersections of each line with $G$ form an infinite sequence of three letters which can be considered as an extension of well-known \emph{Sturmian words}. A classification of such linear ternary sequences is presented and a family of examples is constructed from a notable sequence $S^M$ which could be viewed as an analogue of the \emph{Fibonacci word} in the family of Sturmian words. The factor complexity and the palindromic complexity of these linear ternary sequences are also studied. The last result stated is that each ternary sequence with factor complexity $n+2$ is the intersection sequence of a line.
\end{abstract}

\maketitle

\setlinespacing{1}

\section{Introduction}

A \emph{cutting sequence} is an infinite sequence over the alphabet $A=\{0,1\}$ that exhibits  the meetings of a line $L$ with the vertical and horizontal segments of the square grid of a two dimensional integral lattice in the first quadrant of Euclidean plane. The Cutting sequences were first founded by Christoffel[1] and appeared in works of Andrey Markov [2] for finding all natural integers which satisfies $a^2+b^2+c^2=3abc$. A more geometric point of view is studied by Caroline Series [3]. Cutting sequences can be defined in several ways in geometry, combinatorics, number theory and algebra. For example, consider a binary sequence such that between each pair of factors of the same length, the difference between the number of the occurrences of every digit is at most 1. This represents the cutting sequences as an algebraic object. Another description is that if for each positive integer $n$, the binary sequence $x$ contains exactly $n+1$ different factors with length $n$, $x$ is also a cutting sequence. The sequences from this point of view are also called \emph{Sturmian words}. In fact, Sturmian words are the cutting sequences of lines with irrational slope. These interpretations show that the cutting sequences establish connections between geometry, combinatorics, number theory and algebra.

Sturmian words were first founded by M. Morse and G. A. Hedlund [4] and studied by many others. One can find a survey about this family of words in [5].  One of the most well-known Sturmian words is the \emph{Fibonacci word} which is the cutting sequence of the line with slope $\varphi=\frac{1+\sqrt{5}}{2}$. In this paper, we introduce a family of sequences similar to the cutting sequences which are generated by the lines in the 3-dimensional Euclidean space. This construction gives ternary sequences on the alphabet $T=\{0,1,2\}$. Similar to the results proved about Sturmian words, we will try to formulate some properties about the ternary sequences obtained in this way. We also illustrate a ternary sequence derived from the binary Fibonacci word called $S^M$ which has noticeable  properties commensurable with the Fibonacci word. For example, for each positive integer $n$, the sequence $S^M$ has exactly $n+2$ different factors of length $n$ which is the minimal value in the family of ternary sequences. This sequence is the intersection sequence of a line $L^M$ with azimuthal angle $\theta=\arctan(\varphi)$ and polar angle $\phi= \arccos(\frac{1}{4(\varphi+1)})$.  Furthermore, the slopes of the orthogonal projections of the line $L^M$ in $xy,xz,yz$ planes are $\varphi,\frac{1}{\varphi+1},\varphi$, respectively. This example leads us to a family of linear ternary sequences with the same properties as $S^M$. Then we prove that all of the ternary sequences with the factor complexity $n+2$ are the intersection sequences of lines.

The subjects in this paper appear as follows: in Section \ref{two-d} the two dimensional case is reviewed. Then the three dimensional case is introduced in Section \ref{three-d}. In the last Section, some combinatorial properties of linear ternary intersection sequences are investigated.

\section{Two Dimensional Case} \label{two-d}
Let $D^+$ be the first quadrant of the Euclidean plane and let $G$ denote the integer lattice grid in $D^+$, i.e., $G=\{ax+by \in D^+ \: | \: a,b\in\N\cup \{0\}, x,y\in\mathbb{R}^+, \: a\cdot b=0   \}$. By $L:y=\lambda x$ we denote a line passing through the origin with a positive slope.
\begin{defn}
Let $L:y=\lambda x$ be a line in the first quadrant, and construct the meeting sequence of $L$ with the grid $G$ in a way that for each horizontal meeting we write one 0 and for each vertical meeting we write one 1, sequentially. If $L$  passes through a lattice point, excluding the origin, we write 10 for such a meeting. The corresponding binary sequence shows the {\it cutting sequence} of $L$, and is denoted by $Cs(L)$ or $Cs(\lambda)$.
\end{defn}
\noindent For example, the cutting sequence corresponding to the line $L’:y=\varphi x$ is
\beq
Cs(\varphi)=0100101001001010010\ldots
\eeq
where $\varphi=\frac{1+\sqrt{5}}{2}$ is known as the \emph{golden ratio}. This sequence is called the \emph{Fibonacci word}. One can see that the ratio of appearance of 0's and 1's is related to the slope of the associated line. A cutting sequence can also be presented compactly by replacing each consecutive block of the same digits with a power representation. It is obvious that the cutting sequence of a line with rational slope is periodic and the cutting sequence of a line with irrational slope is aperiodic. For example $Cs(2)=10^210^210^2\ldots$ and $Cs(\varphi)=010^21010^210^21010^210\ldots$.

\subsection{ Derivation of a cutting sequence}
Cutting sequences have geometric and algebraic interpretations. Some concepts which connect these concepts are recalled below.
\bth \label{form}
Let $L:y=\lambda x$ be a line in the first quadrant and suppose $Cs(\lambda)$ is its associated cutting sequence in compact form. If $0<\lambda<1$ and $\frac{1}{\lambda}$ is not an integer, then we have
\begin{eqnarray}
Cs(\lambda)\in\{01^{\lfloor\frac{1}{\lambda}\rfloor},01^{\lfloor\frac{1}{\lambda+1}\rfloor}\}^{\mathbb{N}}.
\end{eqnarray}
 If $\frac{1}{\lambda}$ is an integer, then $Cs(\lambda)$ has the form
 \begin{eqnarray}
 Cs(\lambda)\in\{01^{\lfloor\frac{1}{\lambda}\rfloor}\}^{\mathbb{N}}.
 \end{eqnarray}
 If $\lambda>1$ and $\lambda$ is not an integer, then $Cs(\lambda)$ has the form
\begin{eqnarray}
Cs(\lambda)\in\{10^{\lfloor\lambda\rfloor},10^{\lfloor\lambda+1\rfloor}\}^{\mathbb{N}},
\end{eqnarray}
and If $\lambda$ is an integer, then the form of its cutting sequence is
\begin{eqnarray}
Cs(\lambda)\in\{10^{\lfloor\lambda\rfloor}\}^{\mathbb{N}}.
\end{eqnarray}
By $\lfloor x\rfloor$ we mean the greatest integer number less than or equal to $x$.
\end{thm}
\begin{proof} It is proved by contradiction. One can find the  details in [3].
\end{proof}

\begin{defn}\label{val}
For the line $L:y=\lambda x$, the value of $L$, which is denoted by $val(L)$, is defined as $\lfloor\lambda\rfloor$ if $\lambda>1$, and $\lfloor1/\lambda\rfloor$ if $0<\lambda<1$. The case $\lambda=1$ is a trivial case which will not be considered at all.
\end{defn}

\begin{defn}\label{deriv}
Suppose $L$ is a line with slope $\lambda>1$. Replacing each factor $10^{val(L)}$ with $0^\prime$ and replacing each remaining 1 with $1^\prime$ in $Cs(L)$ results in a new sequence which is called the \emph{derivation of $Cs(L)$} and is denoted by  $Cs^{(1)}(L)$. If a sequence is derived $k$ times, the final sequence is denoted by $Cs^{(k)}(L)$.
\end{defn}

Two important theorems about derivation sequences are presented. Their proofs can be found in [3].
\bth\label{derivation}
For the line $L$ in the plane with slope $\lambda>1$, the derivation sequence $Cs^{(1)}(\lambda)$ is the cutting sequence of a new line $L'$. The slope $\lambda'$ of $L'$ is $\lambda'=\lambda - val(L)$. In the case of $0<\lambda<1$, the associated cutting sequence of $L'$ is $Cs^{(1)}(1/\lambda)$ and its slope is $\lambda'=\frac{1}{\frac{1}{\lambda}-val(L)}$.
\end{thm}
The value of the $k$th derivation of the line $L$ is denoted by $val_k$. It is remarkable that the sequence $Cs(\lambda)$ is equal to $Cs(1/\lambda)$ in which the roles of 0's and 1's are interchanged.

\bth\label{cont-frac}
Let $L$ be a line in $D^+$ and let $V_L=val_0, val_1, val_2,...$ be the sequence of values of the derivations of $L$. Then
\begin{eqnarray}
\lambda=val_0+\frac{1}{val_1+\frac{1}{val_2+\frac{1}{val_3+\frac{1}{val_4+...}}}} .
\end{eqnarray}
\end{thm}
\subsection{Sturmian words and cutting sequences}
A summary of the relationship between the cutting sequences of lines with irrational slopes and other important binary words is presented here.
\begin{defn}
Consider a(n infinite) word $x=x_1x_2\ldots$ . Any finite word $x_{ij}=x_ix_{i+1}...x_j$ of $x$, ($i\leq j$), is a \emph{factor} of $x$ of length $j-i+1$. The set of factors of the word $x$ with cardinality $n$ is denoted by $f_n(x)$.
\end{defn}
\noindent Two factors $x_{ij}$ and $x_{kl}$ of a word $x$ with length $n$ are the same if for each $0\leq t\leq n$, $x_{i+t}=x_{k+t}$.
\begin{defn}
Let $x=x_1x_2x_3\ldots$ be an infinite word. For each $n\in\N$, the \emph{$n$-complexity} of $x$, which is denoted by $C(x,n)$, is the number of different factors of $x$ with length $n$. Clearly, $C(x,n)=\textrm{Card}(f_n(x))$.
\end{defn}
An infinite word $x$ such that $C(x,n)=n+1$ for all $n\in\N$ is called a {\it Sturmian word}.
\begin{defn}
Consider the infinite word $x=x_1x_2x_3\ldots$ . This word is \emph{C-balanced} if for each two factors $u$ and $w$ with equal length, $||w|_a-|u|_a|\leq C$ for all letters $a$, where $|w|_a$ denotes the number of occurrences of $a$ in the
factor $w$.
\end{defn}
The following theorem justifies the equivalency between three concepts in geometry, combinatorics and algebra [3].
\bth\label{equiv}
Let $x=x_1x_2x_3\ldots$ be an infinite word. The following statements are equivalent:

(i) $x$ is the cutting sequence of a line with irrational slope.

(ii) $x$ is a Sturmian word.

(iii) $x$ is a 1-balanced and aperiodic word.
\end{thm}

\section{Three dimensional case}\label{three-d}
In agreement with the two dimensional case, we again use $D^+$ to present the first octant of the Euclidean space and use $G$ for the integral cube grid in $D^+$. By $L=(\theta ,\phi)$ we denote the line passing through the origin with the azimuthal angle $\theta$ and the polar angle $\phi$ in the first octant. Similar to the two dimensional case, we consider {\it the intersection sequence} of $L$ with respect with the cubic grid. In this case, one $0$ is written for each intersection with a face parallel to the $xy$ plane, one $1$ is written for each intersection with a face parallel to the $xz$ plane and one $2$ is written for each intersection with a face parallel to $yz$ plane, so a ternary sequence of letters $\{0,1,2\}$ is obtained. The intersection sequence of the line $L$ is denoted by $Is(L)$. A natural question is whether a ternary sequence is an intersection sequence of a line. To answer this question the following theorem is stated.
\begin{figure}
	\centering
	\includegraphics[width=3.8 cm,height=4.2 cm]{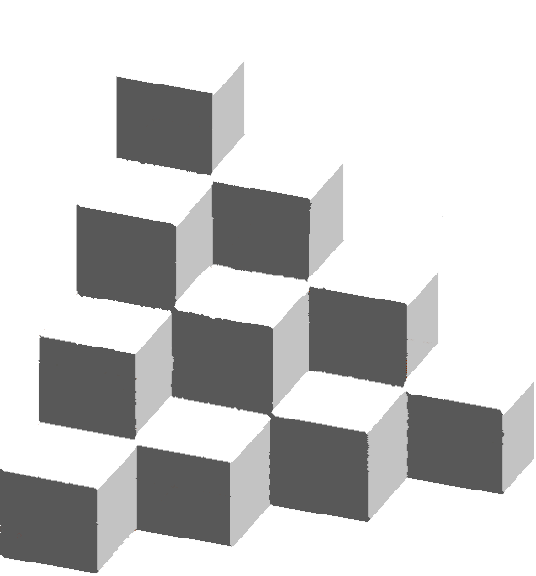}
	\caption{The cube grid}
	\label{grid}
\end{figure}
\bth\label{linearity}
Let $C$ be a line in the $D^+$ and $S$ be the intersection sequence of $C$ with the grid $G$. Then
 by removing each of the 0's, 1's or 2's from $S$, the resulting binary sequence is the cutting sequence of a line in the Euclidean plane.
\end{thm}
\begin{proof}
  The orthogonal projections of $C$ in the two planes $xy$ and $yz$ are lines.
 Now consider the intersection point of $C$ with a face of grid $G$ parallel to the $xy$ plane, which is denoted by $P$. The orthogonal projection $P'$ of $P$ in the $yz$ plane is an intersection point of $C|_{yz}$ with one of the horizontal lines of the square grid of the $yz$ plane. So the binary sequence obtained by removing 2's from the ternary intersection sequence of $C$ is the cutting sequence of the orthogonal projection of $C$ in the square grid of the $yz$ plane. We have the same scenario with $0$'s and $1$'s.
\end{proof}
\subsection{Finding $\theta$ and $\phi$}
\begin{figure}
  \centering
  \includegraphics[width=4 cm,height=4 cm]{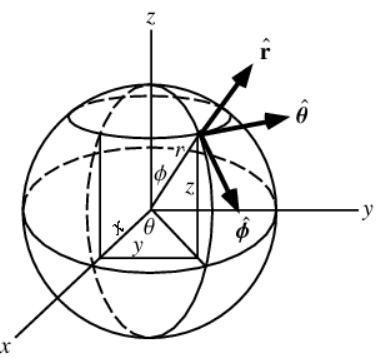}
  \caption{The spherical coordinates}
  \label{spher}
\end{figure}
Consider a line $L$ in $D^+$ with the intersection sequence $Is(L)$. Let $\lambda_{xy}$ be the slope of the line $L|_{xy}$. From Theorem \ref{derivation}, deriving the binary sequence $Cs(L|_{xy})$ results in the value of $\lambda_{xy}$. Then $\theta = \arctan(\lambda_{xy})$. To calculate $\phi$, we have $r^2=x^2+y^2+z^2$ and $\cos(\phi)=\frac{z}{r}$ (see Figure \ref{spher}). Therefore
\begin{eqnarray}
L=\left(\arctan(\lambda_{xy}) , \arccos\left(\sqrt{\frac{1}{{\lambda}^2_{xz}+\frac{1}{{\lambda}^2_{yz}}+1}}\right)\right),
\end{eqnarray}
where $\lambda_{xy}$ is the slope of the line $L|_{xy}$, and $\lambda_{xz}$ is the slope of the line $L|_{xz}$.

For example, consider the simple periodic sequence $T=00120012001200\ldots$. By removing all 2's from $T$, the obtained binary sequence $T_{yz}$ is the cutting sequence of a line with slope $2$. In other words, $T_{yz}=Cs(2)$ and $T_{xz}=Cs(\frac{1}{2})$. Similarly, by removing all 0's from $T$, $T_{xy}=Cs(1)$. So the ternary sequence $T$ is the intersection sequence of the line $L=(\frac{\pi}{4},\arccos(\sqrt{\frac{4}{21}}))$ in the Euclidean space. The following example shows a ternary sequence which is not the intersection sequence of a line.
\begin{exmpl}
\emph{The} Tribonacci \emph{word}
\begin{eqnarray}
t = 010201001020101020100102\ldots
\end{eqnarray}
\emph{is a ternary sequence constructed by the recursive form $t_{n+3}=t_{n+2}t_{n+1}t_{n}$, with $t_1=0,t_2=01,t_3=0102$. This ternary sequence is 2-balanced and for each integer $n>0$, $C(t,n)=2n+1$. Now consider the binary sequence $t_{xz}=0^220^420^320^420^220\ldots$ obtained by removing all $1$'s. By Theorem \ref{form} we know that there is no line with the cutting sequence $t_{xz}$. So The Tribonacci word is not the intersection sequence of any line.} $\square$
\end{exmpl}

\section{Complexity of the intersection sequences of lines} \label{comp-sec}
In [6] it is stated that the intersection sequences of lines and 2-balanced ternary sequences are the same. Unlike the binary case, sequences with minimal complexity and 1-balanced sequences are not the same in the ternary case [5, 7]. But we can construct some 2-balanced ternary sequences with minimal complexity $n+2$, for each integer $n>0$. For example, the novel sequence
\begin{eqnarray}\label{defSM}
S^M=01020101020102010102010102010201010201020101020\ldots
\end{eqnarray}
is a ternary sequence constructed by replacing all $00$ factors in the Fibonacci word with $020$. We have the following interesting property about this sequence.
\bth\label{prjSM}
The ternary sequence $S^M$ is the intersection sequence of a line $L^M$ in which  $Cs(L^M|_{xy})=Cs(\varphi)$, $Cs(L^M|_{xz})=Cs(\frac{1}{\varphi+1})$ and   $Cs(L^M|_{yz})=Cs(\varphi)$.
\end{thm}
\begin{proof}
Removing all $2$'s from $S^M$ we reconstruct the original Fibonacci sequence. Thus $Cs(L^M|_{yz})=Cs(\varphi)$. Removing all $0$'s from $S^M$ results in the sequence of the projection $L^M$ on the $xy$ plane. Applying the following replacements to the Fibonacci word also constructs the $S^M_{xy}$.
\begin{align*}
    01 &\to \textbf{1}\\
    0  &\to \textbf{2}
\end{align*}
These replacements result in a sequence that coincides with the derivation of the Fibonacci word with respect to the factors $01$ and $0$, such that the 2's are replaced by the 0's in turn. So the constructed sequence is equivalent to $Cs^{(1)}(\varphi)$.
By Theorem \ref{derivation}, $Cs(L^M|_{xy})= Cs(\frac{1}{\frac{1}{\varphi}})=Cs(\varphi)$. Removing 1's results in a sequence in which the frequency of 0's is one greater than the frequency of 1's in $S^M_{xy}$. So by Theorem \ref{cont-frac} we have $Cs(L^M|_{xz})=Cs(\frac{1}{\varphi+1})$. 
Since $\lambda_{xz}=\frac{1}{\lambda_{yz}\times\lambda_{xy}}$, there is a line $L^M\in\rl^3$ such that $L^M=(\arctan(\varphi), \arccos(\frac{1}{4(\varphi+1)})$ and $Is(L^M)=S^M$.
\end{proof}
\bth \label{complexity}
For each integer $n>0$, $C(S^M,n)=n+2$.
\end{thm}
\begin{proof}
By the structure of the sequence $S^M$, $S^M_{2k+1}=0$ for  $k\geq0$. We calculate the $n$-complexity of the sequence $S^M$. In $f_n(S^M)$, each factor starting with 0 has $\lceil \frac{n}{2} \rceil$ occurrences of $0$ and $\lfloor \frac{n}{2}\rfloor$ occurrences of $1$ or $2$. Since the binary sequence $S^M_{xy}$ is the cutting sequence of a line with an irrational slope, $C(S^M_{xy},\lfloor\frac{n}{2}\rfloor)=\lfloor\frac{n}{2}\rfloor+1$. Thus, there are $\lfloor\frac{n}{2}\rfloor+1$ factors in $f_n(S^M)$ starting with 0. On the other hand, each factor with 0 as its second letter has $\lfloor \frac{n}{2} \rfloor$ letters of $0$ and $\lceil \frac{n}{2}\rceil$ letters of $1$ or $2$. Since
$C(S^M_{xy},\lceil\frac{n}{2}\rceil)=\lceil\frac{n}{2}\rceil+1$, there are $\lceil\frac{n}{2}\rceil+1$ factors in $f_n(S^M)$ with 0 as the second letter. Therefore, there are $\lfloor\frac{n}{2}\rfloor+1+\lceil\frac{n}{2}\rceil+1=n+2$ factors with length $n$ for each integer $n>0$ in the sequence $S^M$. In other words, for each integer $n>0$, $C(S^M,n)=n+2$.
\end{proof}
\begin{defn}
Let $x=x_1x_2x_3\ldots$ be an infinite sequence over the alphabet $A=\{0,1,2,\ldots,k\}$. The binary sequence obtained by applying the following rules to the sequence $x$ is a \emph{sequential projection} of $x$, which is denoted by $P_ax$.
\begin{center}
\beq
a\rightarrow \textbf{0} &\hbox{, for $a\in A$}\\
b\rightarrow \textbf{1} &\hbox{, for each $b\in A\backslash\{a\}$}
\eeq
\end{center}
\end{defn}

\bth\label{seq-prj}
Consider the alphabet $A=\{0,1,2\}$ and let $S=x_1x_2x_3\ldots$ be an infinite sequence over the alphabet $A$ such that for each positive integer $n$, $C(S,n)=n+2$. Suppose there is an $a\in A$ such that $P_aS=Cs(1)$. Then the binary sequence obtained by removing $a$'s from the sequence $S$ is cutting sequence of a line.
\end{thm}
\begin{proof}
Without loss of generality, suppose $a=0$. The occurrence of 0's in the sequence $S$ is periodic but the factor complexity function of this sequence is strictly increasing. So, the binary sequence $S_{xy}$ is aperiodic and for each positive integer $n$, $C(S_{xy},n)\geq n+1$. Applying the same strategy as the proof of Theorem \ref{complexity} results in $C(S,n)=C(S_{xy},\lfloor\frac{n}{2}\rfloor)+C(S_{xy},\lceil\frac{n}{2}\rceil)=n+2$. By the previous argument we should have  $C(S_{xy},\lfloor\frac{n}{2}\rfloor)=\lfloor\frac{n}{2}\rfloor+1$ and $C(S_{xy},\lceil\frac{n}{2}\rceil)=\lceil\frac{n}{2}\rceil+1$. By Theorem \ref{equiv}, the binary sequence $S_{xy}$ is the cutting sequence of a line.
\end{proof}

\begin{defn}
Let $x=x_1x_2x_3\ldots x_n$ be a finite sequence. The sequence $x^r=x_nx_{n-1}x_{n-2}\ldots x_1$ is the \emph{reverse} of $x$.
\end{defn}
It is proved in [4] that for any finite sequence $u\in f_n(x)$ where $x$ is Sturmian, $u^r\in f_n(x)$.
\bth
Let the finite sequence $u$ be a factor of the infinite ternary sequence $S^M$ defined in (\ref{defSM}). Then the sequence $u^r$ is a factor of $S^M$.
\end{thm}
\begin{proof}
By Theorem \ref{prjSM} the binary sequence $S^M_{xy}$ is Sturmian. Thus, if $u_{xy}\in f_n(S^M_{xy})$, then $u^r_{xy}\in f_n(S^M_{xy})$. Since  $S^M_{2k+1}=0$ for each integer $k\geq0$, the ternary sequence $u'=0(u^r_{xy})_10(u^r_{xy})_20\ldots 0(u^r_{xy})_n0$ is a factor of the infinite ternary sequence $S^M$. Since $u^r$ is a factor of $u'$, $u^r$ is a factor of $S^M$.
\end{proof}
\begin{defn}
Let $x=x_1x_2x_3\ldots x_n$ be a finite sequence. The sequence $x$ is \emph{palindrome} if and only if for each integer $k>0$, $x_k=x_{n-k+1}$. Clearly, $x$ is palindrome if and only if $x=x^r$.
\end{defn}
\begin{defn}
Let $x=x_1x_2x_3\ldots$ be an infinite sequence. For each integer $n>0$, the function $P(x,n)$ counts the number of factors of $x$ with length $n$ which are palindromes, i.e., $P(x,n)=Card(\{u\in f_n(x)|u=u^r\})$.
\end{defn}
\begin{lem}\label{palindrome}
Let $x=x_1x_2x_3\ldots$ be an infinite binary sequence. Then $x$ is Sturmian if and only if  for each integer $k\geq0$, $P(x,2k+1)=2$ and $P(x,2k)=1$
\end{lem}
One can find a proof of this lemma in [4]. We want to state a similar classification in the case of ternary sequences with a sequential projection of slope $1$.
\bth\label{palindromic-comp}
For each integer $k\geq 0$, $P(S^M,2k)=0$ and $P(S^M,2k+1)=3$.
\end{thm}
\begin{proof}
We compute the palindromic complexity in two cases. In the first case, we have $P(S^M,2)=0$. So for each integer $k>0$, $P(S^M,2k)=0$. In the second case, $P(S^M,1)=3$. For $k>0$, since the parity of the places of all 0 letters are equal, the appearance of 0's in the odd length factors is symmetric. So to check that the factor $u\in f_{2k+1}(S^M)$ is a palindrome, it is enough to check if the binary sequence $u_{xy}$ is palindrome. As in the proof of Theorem \ref{complexity}, there are $\lfloor\frac{2k+1}{2}\rfloor+1$ factors with length $2k+1$ with the first letter $0$ and there are $\lceil\frac{2k+1}{2}\rceil+1$ factors with length $2k+1$ with the second letter $0$. So, $P(S^M,2k+1)=P(S^M_{xy},\lfloor\frac{2k+1}{2}\rfloor)+P(S^M_{xy},\lceil\frac{2k+1}{2}\rceil)$. Since the parity of $\lfloor\frac{2k+1}{2}\rfloor+1$ and $\lceil\frac{2k+1}{2}\rceil+1$ is always different, then for each integer $k\geq0$, $P(S^M,2k+1)=3$, by Lemma \ref{palindrome}.
\end{proof}

\begin{cor}
Let $x=x_1x_2x_3\ldots$ be an aperiodic infinite ternary sequence over alphabet $A=\{0,1,2\}$. Suppose there is an $a\in A$ such that $P_ax=Cs(1)$. Then $P(x,2k)=0$ and $P(x,2k+1)=3$, for $k=0,1,2,\ldots$ .
\end{cor}
\begin{proof}
Without loss of generality we suppose that $a=0$. Since $P_0x$ has slope 1, letters of 0 and 1 appear alternatingly in $P_0x$. From now on, the technique of the proof of the Theorem \ref{palindromic-comp} can be applied to the sequence $P_0x$ instead of $S^M$.
\end{proof}

\begin{defn}
Let $x=x_1x_2x_3...$ be an infinite word in the alphabet $A=\{0,1,\ldots,k\}$. The (directed) Rauzy graph $G_n(x)$ of order $n$ is a (directed) graph with vertex set $F_n(x)$ and the edge set
\beq
\{(bv,va)| a,b\in A, bva\in f_{n+1}(x)\}
\eeq
\end{defn}

We want to classify the Rauzy graphs of linear intersection sequences with complexity $n+2$.
\bth\label{Rauzy}
Consider the Rauzy graph $G_n(S^M)$. The (directed) degree of all vertices of $G_n(S^M)$ is $1$ except one vertex with the (directed) degree $2$.
\end{thm}
\begin{figure}
  \centering
 \includegraphics[width=4 cm,height=4 cm]{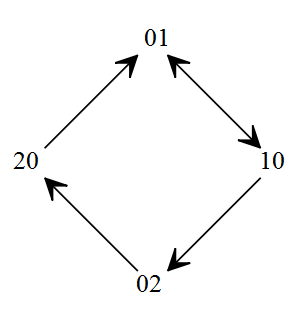}
 \caption{The graph $G_2(S^M)$.}
 \label{fig}
\end{figure}
\begin{proof}
Since the ternary sequence $S^M$ is infinite, for each factor $d\in f_n(S^M)$ there is a letter $a\in\{0,1,2\}$ for which the factor $d'=da$ is in $f_{n+1}(S^M)$. So the degree of each vertex of the graph $G_n(S^M)$ is at least 1. By Theorem \ref{complexity}, $C(S^M,n)=n+2$ so $C(S^M,1)=3$ and for each integer $n>1$, $C(S^M,n)=C(S^M,n-1)+1$. So there should be exactly one factor in $f_{n-1}(S^M)$ that generates two factors of $f_n(S^M)$. Thus there is a unique vertex $w\in G_n(S^M)$ with degree 2. The graph $G_2(S^M)$ is shown in the Figure \ref{fig}.
\end{proof}

Now we want to prove a generalization of Theorem \ref{seq-prj} in order to classify the linearity of the all ternary sequences with minimal complexity.
\bth
Let $S=x_1x_2x_3...$ be an infinite sequence such that for each positive integer $n$, $C(S,n)=n+2$. Then $S$ is an intersection sequence of a line.
\end{thm}
\begin{proof}
Two cases are considered. Suppose there exists a letter $a\in\{0,1,2\}$ such that $P_aS=Cs(1)$. Without loss of generality we can assume that $a=0$. Removing 0's from $S$ results in the sequence $S_{1,2}$. By Theorem \ref{seq-prj}, $C(S_{1,2},n)=n+1$, for each positive integer $n$, and so by Theorem \ref{equiv} there exists a line $\ell_{xy}$ with irrational slope $\lambda_{xy}$ which  the sequence $S_{1,2}$ is its cutting sequence. 
From Theorem \ref{form} and Definition \ref{val},  $S_{1,2}\in\{12^{val(\ell_{xy})},12^{val(\ell_{xy})+1}\}^\N$ or $S_{1,2}\in\{21^{val(\ell_{xy})},21^{val(\ell_{xy})+1}\}^\N$.  By the symmetry we can assume that the first representation occurred. Since 0 letters in $S$ appear alternatingly, the number of 2's between two consecutive 1's is one less than the number of 0's. So each factor of $S$ between two consecutive $1$'s contains $val(\ell_{xy})+1$ or $val(\ell_{xy})+2$  letters of $0$ with respect to the occurrence of $2$ in that factor, by Definition \ref{deriv} one can asserts that the first derivation of $S_{1,2}$ and $S_{0,1}$ are the same. 
Therefore by Theorem \ref{cont-frac} the line $\ell_{yz}$ with slope $\lambda_{yz}=\frac{1}{\lambda_{xy}+1}$ has $S_{0,1}$ as its cutting sequence. By a similar reasoning it is deduced that the line $\ell_{xz}$ with the slope $\lambda_{xz}=1+\frac{1}{\lambda_{xy}}$ has $S_{0,2}$ as its cutting sequence. Because $\lambda_{xy}=\frac{1}{\lambda_{xz}\times \lambda_{yz}}$, there exists a line $L\s\rl^3$ such that its orthogonal projections on $xy$, $xz$ and $yz$ planes are $\ell_{xy}$, $\ell_{xz}$ and $\ell_{yz}$, respectively. If we define $S'=x_1'x_2'x_3'\ldots$ as the intersection sequence of $L$ in $\rl^3$ then by a simple contradiction one can show that $S=S'$. For this, let $k$ be the first place in which $x_k\neq x_k'$  and consider the sequence $S_{x_k,x_k'}$ and $S'_{x_k,x_k'}$. By the above construction of $S'$, we should have $S_{x_k,x_k'}=S'_{x_k,x_k'}$ which contradicts the difference of $S$ and $S'$ in the $k$th letter.

If there is no $a\in\{0,1,2\}$ such that $P_ax=Cs(1)$, then there exists an $A\in\{0,1,2\}$ such that $AA\in f_2(x)$. Renaming $\{A,B,C\}=\{0,1,2\}$, the graph $G_2(S)$ has the form shown in Figure \ref{fig2}.
\begin{figure}
\centering
\includegraphics[width=4 cm,height=5 cm]{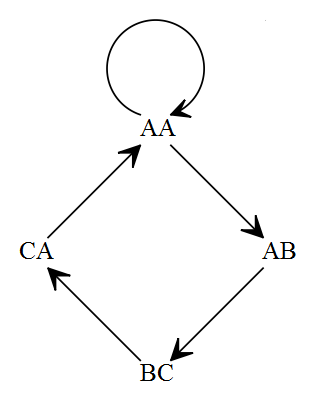}
\caption{The graph $G_2(S)$}
\label{fig2}
\end{figure}
So there is an integer $n\in\N$ such that the sequence $S$ has the following form
\begin{eqnarray*}
S\in\{A^nBC,A^{n+1}BC\}^{\mathbb{N}}
\end{eqnarray*}
Removing all $A$ letters from $S$ results in the cutting sequence of the line $\ell_{xy}:y=x$. On the other hand, removing $C$ from $S$ results in the sequence $S_{A,B}$ which is the cutting sequence of a line with slope $\lambda$, by Theorem \ref{equiv}. Similarly, $S_{A,C}=Cs(\lambda)$. So the ternary sequence $S$ is the intersection sequence of the line $L\s\rl^3$ with slopes $1$, $\lambda$ and $\frac{1}{\lambda}$. So a similar situation occurs as the previous case.
\end{proof}

\noindent {\bf Acknowledgment} - The authors would like to thank Professor P. Arnoux for very helpful communications.
\vskip .8 cm
\setlinespacing{1.3}
\noindent {\bf References}

\noindent [1] \textsc{Christoffel}, E. B. ; \emph{Observatio arithmetica}, Annali di Matematica 6, 1875,145-152.

\noindent [2] \textsc{Markoff}, A. A. ; \emph{Sur les formes quadratiques binaires indefinies}, Mathematische Annalen 15, 1879, 381-496;

\noindent [3] \textsc{Series}, C ; The Geometry of Markoff Numbers, \emph{The Mathematical Intelligencer} Vol 7, No 3. 1985.

\noindent[4] \textsc{Morse}, M. and \textsc{Hedlund}, G. A. ; \emph{Symbolic dynamics}, Amer. J. Math. 60 (1938) 815-866.

\noindent [5] \textsc{Berstel}, J. ; \emph{Recent results on extensions of Sturmian words}, International Journal of Algebra and Computation, Vol. 12 (2002) 371-385.

\noindent [6] \textsc{Durand}, F. and \textsc{Guerziz}, A. and \textsc{koskas}, M. ; Words and morphisms with Sturmian erasures, \emph{Bull. Belg. Math. Soc. Simon Stevin}, 11, No. 4 (2004), 575-588.

\noindent [7] \textsc{Balkov\'a, L\'ubomra, Pelantov\'a, Edita, and Starosta, \v{S}t\v{e}p\'an} ; Sturmian jungle (or garden?) on multiliteral alphabets,  \it{RAIRO - Theoretical Informatics and Applications}, 44, No. 4 (2011), 443-470.

\vskip .5 cm

{\it Email:}
 \verb"saleh"@ \verb"helli.ir "

\hskip 1.2 cm \verb"jahangiri"@ \verb"ipm.ir"
\end{document}